\documentclass{amsart}
\usepackage[margin=1.00in]{geometry}

\usepackage{amsfonts}
\usepackage{setspace}
\usepackage{graphicx}

	\usepackage{mathrsfs}
\usepackage{hyperref}
\usepackage{graphicx}
\usepackage{amsmath, amsthm, amscd, amsfonts, amssymb, graphicx, color}
 \usepackage{url}
\usepackage{enumerate}
 \makeatletter
\let\reftagform@=\tagform@
\def\tagform@#1{\maketag@@@{(\ignorespaces\textcolor{blue}{#1}\unskip\@@italiccorr)}}
\renewcommand{\eqref}[1]{\textup{\reftagform@{\ref{#1}}}}
\makeatother
\usepackage{hyperref}
\hypersetup{colorlinks=true, linkcolor=red, anchorcolor=green,
citecolor=cyan, urlcolor=red, filecolor=magenta, pdftoolbar=true}

\usepackage{epsfig}        
\usepackage{epic,eepic}       

\newtheorem{theorem}{Theorem}
\theoremstyle{plain}

\newtheorem{corollary}{Corollary}

\newtheorem{lemma}{Lemma}

\newtheorem{remark}{Remark}

\numberwithin{equation}{section}

\DeclareMathOperator{\spe}{sp}
\DeclareMathOperator{\re}{Re}   
\DeclareMathOperator{\im}{Im}

\def\etal{et al.\,}

\begin{document}
	\title[The generalized Schwarz inequality for semi-Hilbertian space operators]{The generalized Schwarz inequality for semi-Hilbertian space operators and Some $A$-numerical radius inequalities}
	\author[M.W. Alomari]{Mohammad W. Alomari}
	
	\address {Department of Mathematics, Faculty of Science and Information	Technology, Irbid National University,  P.O. Box 2600, Irbid, P.C. 21110, Jordan.}
	\email{mwomath@gmail.com}

	\date{\today}
	\subjclass[2010]{Primary: 47A12, 47A63, 47B65.  Secondary:  15A18  , 15A45.}
	
	\keywords{Positive operator, Semi-inner product, The mixed Schwarz inequality, $A$-numerical radius}

	\begin{abstract}
	   In this work, the mixed Schwarz inequality for semi-Hilbertian space operators is proved. Namely, for every positive Hilbert space operator 
	  $A$.  If $f$ and $g$ are nonnegative continuous functions on $\left[0,\infty\right)$ satisfying $f(t)g(t) =t$ $(t\ge0)$, then
	   \begin{align*}
	   \left| {\left\langle {T   x,y} \right\rangle_A } \right|   \le 
	   \left\| {f\left( {\left| T \right|_A x} \right)} \right\|_A \left\|  {g\left( {\left| {T^{\sharp_A } } \right|_A y} \right)} \right\|_A
	   \end{align*}
	   for every Hilbert space operator  $T$ such that  the range of $T^* A$ is a subset in the range of  $A$,  such that $A$ commutes with $T$,  and for all vectors $x,y\in \mathscr{H}$, where $\left| T \right|_A = \left(AT^{\sharp_A}T\right)^{1/2}$ such that $T^{\sharp_A}=A^{\dag}T^*A$, where $A^{\dag}$  is the Moore-Penrose inverse of $A$.   Based on that, some inequalities for the $A$-numerical radius are introduced.
	\end{abstract}
	
	\maketitle

\section{Introduction} 

 Let $\mathscr{B}\left( \mathscr{H}\right) $ be the Banach algebra
 of all bounded linear operators defined on a complex Hilbert space
 $\left( \mathscr{H};\left\langle \cdot ,\cdot \right\rangle
 \right)$  with the identity operator  $1_\mathscr{H}$ in
 $\mathscr{B}\left( \mathscr{H}\right) $. 
 For $A\in \mathscr{B}\left( \mathscr{H}\right)$ we denote by $\mathcal{R}\left(A\right)$ and $\mathcal{N}\left(A\right)$ the range and the null
 space of $A$, respectively. By $\overline{\mathcal{R}\left(A\right)}$ we denote the norm closure of  $\mathcal{R}\left(A\right)$. Let $T^*$ be
 the adjoint of $T$ . The cone of all positive (semidefinite) operators is given by
 \begin{align*}
 \mathscr{B}^+\left( \mathscr{H}\right) = \left\{
 A \in \mathscr{B}\left( \mathscr{H}\right) :\left\langle {Ax,x} \right\rangle  \ge 0,\forall x \in \mathscr{H}\right\}.
 \end{align*}
Every $A\in  \mathscr{B}^+\left( \mathscr{H}\right)$ defines the following positive semidefinite sesquilinear form: 
\begin{align*}
 \left\langle { \cdot , \cdot } \right\rangle _A :\mathscr{H} \times\mathscr{H} \longrightarrow \mathbb{C},\left( {x,y} \right) \mapsto \left\langle {x,y} \right\rangle _A  = \left\langle {Ax,y} \right\rangle.
\end{align*} 
The seminorm induced by this sesquilinear form is given by $\left\|x\right\|_A=\sqrt{ \left\langle {x,x} \right\rangle_A }$. It is well-known that $\left\|x\right\|_A$  is a norm on $ \mathscr{H}$ if and only if $A$ is injective and
 $\left( \mathscr{H},  \left\|  \cdot  \right\|_A\right)$ is complete if and only if $\mathcal{R}\left(A\right)$ is closed in $\mathscr{H}$.
 
 An operator $S\in \mathscr{B}\left( \mathscr{H}\right)$ is said to be $A$-adjoint of an operator $T\in \mathscr{B}\left( \mathscr{H}\right)$ if $\left\langle {Tx,y} \right\rangle_A  = \left\langle {x,Sy} \right\rangle_A$. In other words, $S$ is an $A$-adjoint of $T$ if and only if $S$ is a solution of the equation  $AX = T^*A$ in $\mathscr{B}\left( \mathscr{H}\right)$. For $T\in \mathscr{B}\left( \mathscr{H}\right)$ the existence of an $A$-adjoint of $T$ is not guaranteed. The set of all operators acting on $\mathscr{H}$ that admit $A$-adjoints is denoted
 by $\mathscr{B}_A\left( \mathscr{H}\right)$. The existence of such set of operators is guaranteed by Douglas theorem \cite{D} that
 \begin{align*}
 \mathscr{B}_A \left( \mathscr{H} \right) = \left\{ {T \in  \mathscr{B} \left( \mathscr{H} \right):\mathcal{R}\left( {T^* A} \right) \subseteq \mathcal{R}\left( A \right)} \right\}.
 \end{align*}
Moreover, if $T\in $ then the operator equation $AX =T^*A$
has a unique solution, denoted by $T^{\sharp_A}$, satisfying $\mathcal{R}\left( {T^{\sharp_A}} \right) \subseteq \overline{\mathcal{R}\left( A \right)}$ and $\mathcal{N}\left(T^{\sharp_A}\right)\subseteq \mathcal{N}\left(T^*A\right)$. The distinguished $A$-adjoint operator of $T$ or simply $T^{\sharp_A}$ and can be computed as $T^{\sharp_A}=A^{\dag}T^*A^*$,  and satisfy the equation $AT^{\sharp_A}=T^*A$, where $A^{\dag}$  is the Moore-Penrose inverse of $A$ (see \cite{ACG1} and \cite{ACG2}).
  
Denotes $\mathscr{B}_{A^{1/2}} \left( \mathscr{H} \right)$ the set of all operators $T\in \mathscr{B}  \left( \mathscr{H} \right)$
such that $T$ is bounded induced by the semi-norm $\left\|\cdot\right\|_A$. In other words,
\begin{align*}
\mathscr{B}_{A^{1/2}} \left( \mathscr{H} \right):= \left\{T\in \mathscr{B}  \left( \mathscr{H} \right): \left\|Tx\right\|_A\le \lambda \left\|x\right\|_A, \text{for some $\lambda >0$ and all} \,\, x\in \mathscr{H}\right\}.
\end{align*}
Members of $\mathscr{B}_{A^{1/2}} \left( \mathscr{H} \right)$ are called $A$-bounded operators. In fact, if $T\in \mathscr{B}_{A^{1/2}} \left( \mathscr{H} \right)$, then the $A$-operator seminorm is defined as: 
\begin{align*}
\left\|T\right\|_A=\mathop {\sup }\limits_{\scriptstyle x \in \overline {\mathcal{R}\left( A \right)}  \hfill \atop \scriptstyle \,\,\,\,\,x \ne 0 \hfill} \frac{{\left\| {Tx} \right\|_A }}{{\left\| x \right\|_A }} = \sup \left\{ {\left\| {Tx} \right\|_A :x \in \mathscr{H},\left\| x \right\|_A  = 1} \right\}.
\end{align*}
It is convenient to note that it may happen that $\left\|T\right\|_A=+\infty$ for some operator $T\in \mathscr{B} \left( \mathscr{H} \right)\backslash \mathscr{B}_{A^{1/2}} \left( \mathscr{H} \right)$. Also, we need to mention that the inclusions
\begin{align*}
\mathscr{B}_{A} \left( \mathscr{H} \right) \subset \mathscr{B}_{A^{1/2}} \left( \mathscr{H} \right) \subset \mathscr{B} \left( \mathscr{H} \right)
\end{align*}
hold with equality if $A$ is injective and has closed range. But neither $\mathscr{B}_{A} \left( \mathscr{H} \right) $ nor $\mathscr{B}_{A^{1/2}} \left( \mathscr{H} \right)$ is closed and dense in $\mathscr{B} \left( \mathscr{H} \right)$.

In particular, we should note that if $T\in \mathscr{B}_{A} \left( \mathscr{H} \right) $ then $T^{\sharp_A} \in \mathscr{B}_{A} \left( \mathscr{H} \right) $ and $\left(T^{\sharp_A}\right)^{\sharp_A}=P_ATP_A$, where $P_A$ denotes the orthogonal projection onto $\overline{\mathcal{R}\left(A\right)}$. Moreover,  we have
\begin{align*}
\left\|T^{\sharp}\right\|_A=\left\|T\right\|_A.
\end{align*} 

An operator $T\in \mathscr{B}_A\left( \mathscr{H} \right) $ is called $A$-selfadjoint if $AT$ is selfadjoint, i.e., $AT=T^*A$, or simply $
\left\langle {Tx,x} \right\rangle_A \in \mathbb{R}$ for all $x\in
\mathscr{H}$. Also, $T$ is called $A$-positive if $AT$ is positive $AT>0$. Note that if $T$ is $A$-selfadjoint then $T \in \mathscr{B}_{A} \left( \mathscr{H} \right)$. An operator $T \in \mathscr{B}_{A} \left( \mathscr{H} \right)$ is said to be $A$-normal if $T T^{\sharp_A} = T^{\sharp_A}T$. The fact that every selfadjoint operator is normal does not hold in this case; i.e., an $A$-selfadjoint operator
is not necessarily $A$-normal (see \cite[Example 5.1]{BFA}). Indeed, this property holds if $T$ commutes with $A$. We note that, for any $T\in \mathscr{B}_A\left( \mathscr{H} \right) $ the $A$-Cartesian decomposition is given by $T=\re_A\left(T\right)+i\im_A\left(T\right)$, where $\re_A\left(T\right)=\frac{T+T^{\sharp_A}}{2}$ and $\im_A\left(T\right)=\frac{T-T^{\sharp_A}}{2i}$. Moreover, $\re_A\left(T\right)$ and $\im_A\left(T\right)$ are $A$-selfadjoint operators.\\

In 2012, Saddi \cite{S}, introduced the definition of $A$-spectral radius as follows:
\begin{align}
\label{eq1.1}r_A\left(T\right)= \mathop {\lim }\limits_{n \to \infty } \sup \left\| {T^n } \right\|_A^{\frac{1}{n}}. 
\end{align}
But indeed, this formula was recently proved by Feki in \cite{Kais3}, where he gave  a counterexample showing that the definition of Saddi in \cite{S} doesn't guarantee that
$r_A\left(T\right) < \infty$. 
The Feki definition of $A$-spectral radius  reads: 
\begin{align}
\label{eq1.2}r_A\left(T\right)=
\mathop {\inf }\limits_{n \in N} \left\| {T^n } \right\|_A^{\frac{1}{n}}.
\end{align}

For $A$-bounded linear operator $T$ on a Hilbert space
$\mathscr{H}$, the $A$-numerical range $W_A\left(T\right)$ is the image
of the  unit sphere of $\mathscr{H}$ under the positive semidefinite sesquilinear quadratic form $x\to
\left\langle {Tx,x} \right\rangle_A$ associated with the operator.
More precisely,
\begin{align*}
W_A\left( T \right) = \left\{ {\left\langle {Tx,x} \right\rangle_A :x
	\in \mathscr{H},\left\| x \right\|_A = 1} \right\}
\end{align*}
Also, the $A$-numerical radius is defined to be
\begin{align*}
w_A\left( T \right) = \sup \left\{ {\left| \lambda\right|:\lambda
	\in W_A\left( T \right) } \right\} = \mathop {\sup }\limits_{\left\|
	x \right\|_A = 1} \left| {\left\langle {Tx,x} \right\rangle_A }
\right|.
\end{align*}
For more about properties of $A$-numerical range and $A$-numerical radius, see \cite{BFA}--\cite{BKP2}, \cite{Kais1}, \cite{MXZ}, \cite{RSM}, and \cite{Z}.

Recently, it was shown that the inequality \cite{Kais3} (see also \cite{Kais4})
\begin{align*}
r_A\left(T\right) \le w_A\left(T\right) \le \left\|T\right\|_A
\end{align*}
for any $T\in \mathscr{B}_{A^{1/2}}\left(\mathscr{H}\right)$. Also, $\left\|\cdot\right\|_A$ and $w_A\left(T\right)$ are equivalent seminorm on $  \mathscr{B}_{A^{1/2}}\left(\mathscr{H}\right)$ satisfying the inequality:
\begin{align*}
\frac{1}{2}\left\|T\right\|_A \le w_A\left(T\right) \le \left\|T\right\|_A.
\end{align*}
The first inequality becomes equality if $AT^2 = 0$ and the second inequality becomes equality if $T$ is $A$-normal (see \cite{Kais3}).

The Schwarz inequality for positive operators reads that if $A$ is a positive operator in $\mathscr{B}\left(\mathscr{H}\right)$, then
\begin{align}
\left| {\left\langle {Ax,y} \right\rangle} \right|  ^2  \le \left\langle {A x,x} \right\rangle \left\langle { A y,y} \right\rangle  \label{eq1.4}
\end{align}
for any   vectors $x,y\in \mathscr{H}$.

In 1951, Reid \cite{R} proved an inequality which in some senses
considered a variant of the Schwarz inequality. In fact, he proved
that for all operators $A\in \mathscr{B}\left( \mathscr{H}\right)
$ such that $A$ is positive and $AB$ is selfadjoint then
\begin{align}
\left| {\left\langle {ABx,y} \right\rangle} \right|  \le \|B\|
\left\langle {A x,x} \right\rangle, \label{eq1.5}
\end{align}
for all $x\in \mathscr{H}$. In \cite{H}, Halmos presented his
stronger version of the Reid inequality \eqref{eq1.5} by replacing
$r\left(B\right)$ instead of $\|B\|$.

In 1952, Kato  \cite{TK} introduced a companion inequality of
\eqref{eq1.4}, called  the mixed Schwarz inequality,  which
asserts
\begin{align}
\left| {\left\langle {Ax,y} \right\rangle} \right|  ^2  \le \left\langle {\left| A \right|^{2\alpha } x,x} \right\rangle \left\langle {\left| {A^* } \right|^{2\left( {1 - \alpha } \right)} y,y} \right\rangle, \qquad 0\le \alpha \le 1. \label{eq1.6}
\end{align}
for every   operators $A\in \mathscr{B}\left( \mathscr{H}\right) $ and any vectors $x,y\in \mathscr{H}$, where  $\left|A\right|=\left(A^*A\right)^{1/2}$.

In 1988,  Kittaneh  \cite{FK2} proved  a very interesting extension combining both the Halmos--Reid inequality \eqref{eq1.2} and the mixed Schwarz inequality \eqref{eq1.6}. His result reads that
\begin{align}
\left| {\left\langle {ABx,y} \right\rangle } \right| \le r\left(B\right)\left\| {f\left( {\left| A \right|} \right)x} \right\|\left\| {g\left( {\left| {A^* } \right|} \right)y} \right\|\label{eq1.7}
\end{align}
for any   vectors $x,y\in  \mathscr{H} $, where $A,B\in \mathscr{B}\left( \mathscr{H}\right)$ such that $|A|B=B^*|A|$ and    $f,g$ are  nonnegative continuous functions  defined on $\left[0,\infty\right)$ satisfying that $f(t)g(t) =t$ $(t\ge0)$.       Clearly, choose $f(t)=t^{\alpha}$ and $g(t)=t^{1-\alpha}$ with   $B=1_{\mathscr{H}}$ we refer to \eqref{eq1.6}. Moreover, choosing $\alpha=\frac{1}{2}$ some manipulations refer to the Halmos version of the Reid inequality.\\

In this work,  the corresponding version of the well-known Kittaneh  inequality \eqref{eq1.7}, which is also known as the mixed Schwarz inequality for the semi-Hilbertain space operators is introduced. Based on that, some inequalities for the $A$-numerical radius are proved. A generalization of the Euclidean operator $A$-radius with some basic properties are discussed and elaborated. A generalization of an important inequality proved recently by Feki in \cite{Kais2} for the generalized Euclidean operator $A$-radius is also considered.

\section{Preliminaries and Lemmas}

 The corresponding version of Schwarz inequality for $A$-positive operators reads that if $T$ is $A$-positive operator in $\mathscr{B}_A\left(\mathscr{H}\right)$, then
 \begin{align}
 \label{eq2.1}\left| {\left\langle {Tx,y} \right\rangle_A} \right|  ^2  \le \left\langle {T x,x} \right\rangle_A \left\langle { T y,y} \right\rangle_A
 \end{align}
 for any   vectors $x,y\in \mathscr{H}$. The proof of this result can be done using the same argument of the proof of the classical Schwarz inequality for positive operators taking into account that we use the semi-inner product induced by $A\in \mathscr{B}^+\left(\mathscr{H}\right)$.
 
  In order to introduce the corresponding version of the mixed Schwarz inequality (..) for semi-Hilbertian space oprtators we need the following sequence of lemmas; which have been pointed out in ...  but for general Hilbert space operators. We rewrite these lemmas in appropriate way for semi-Hilbertian space oprtators. 

 \begin{lemma}
 \label{lemma1}	Let $A\in \mathscr{B}^+\left(\mathscr{H}\right)$. Let $T,S$ and $R$ be operators in $\mathscr{B}_{A }\left(\mathscr{H}\right)$, where $T$ and $S$ are $A$-positive. Then $\left[ {\begin{array}{*{20}c}
 		T & {R^{\sharp_A } }  \\
 		R & S  \\
 		\end{array}} \right]$ is  ${\bf{A}}$-positive operator in $\mathscr{B}_{{\bf{A}} }\left(\mathscr{H}\oplus\mathscr{H}\right)$ if and only if $\left| {\left\langle {Rx,y} \right\rangle _A } \right|^2  \le \left\langle {Tx,x} \right\rangle _A \left\langle {Sy,y} \right\rangle _A$ for all $x,y\in \mathscr{H}$, where $
 	{\bf{A}} = \left[ {\begin{array}{*{20}c}
 		A & 0  \\
 		0 & A  \\
 		\end{array}} \right]\in \mathscr{B}^+\left(\mathscr{H}\oplus\mathscr{H}\right)$.
 \end{lemma}
 \begin{proof}
 	The proof is straightforward by replacing the inner product 
 	$\left\langle { \cdot , \cdot } \right\rangle$ by the semi-inner product $\left\langle { \cdot , \cdot } \right\rangle_A$ and setting $A=T$, $B=S$, $C=R$ and $C^*=R^{\sharp_A}$,  in \cite[Lemma 1]{FK2}.
 \end{proof}
 
 \begin{lemma}
  \label{lemma2}	Let $A\in \mathscr{B}^+\left(\mathscr{H}\right)$. Let $T,S$ and $R$ be operators in $\mathscr{B}_{A }\left(\mathscr{H}\right)$, where $T$ and $S$ are $A$-positive and $SR=RT$. If $\left[ {\begin{array}{*{20}c}
 		T & {R^{\sharp_A } }  \\
 		R & S  \\
 		\end{array}} \right]$ is  ${\bf{A}}$-positive operator in $\mathscr{B}_{{\bf{A}} }\left(\mathscr{H}\oplus\mathscr{H}\right)$, then $
 	\left[ {\begin{array}{*{20}c}
 		{f^2 \left( T \right)} & {R^{\sharp_A } }  \\
 		R & {g^2 \left( S \right)}  \\
 		\end{array}} \right]$ is also ${\bf{A}}$-positive,    where where $f$ and $g$ are non-negative functions on $ \left[ {0,\infty } \right)$ which are continuous
 	and satisfying the relation $f(t)g(t)=t$ for all $t\in \left[ {0,\infty } \right)$,  and $
 	{\bf{A}} = \left[ {\begin{array}{*{20}c}
 		A & 0  \\
 		0 & A  \\
 		\end{array}} \right]\in \mathscr{B}^+\left(\mathscr{H}\oplus\mathscr{H}\right)$.
 \end{lemma}

 \begin{proof}
 	The proof is straightforward by replacing the inner product 
 	$\left\langle { \cdot , \cdot } \right\rangle$ by the semi-inner product $\left\langle { \cdot , \cdot } \right\rangle_A$ and setting $A=T$, $B=S$, $C=R$ and $C^*=R^{\sharp_A}$, in \cite[Lemma 2]{FK2}.
 \end{proof}

It is well known that for every selfadjoint operator $T\in \mathscr{B} \left(\mathscr{H}\right)$ the inequality 
\begin{align*}
\left|\left\langle {Tx,x} \right\rangle\right|  \le  \left\langle {\left|T\right| x,x} \right\rangle 
\end{align*}
holds for every vector $x\in \mathscr{H}$, where $\left|T\right|=\left(T^*T\right)^{1/2}$. 
To find out what are the appropriate conditions required for generalizing the previous inequality and in lighting of what was discussed previously,  neither selfadjoint operatos nor $A$-selfadjoint operator will be useful in this case. In fact, we need a new type of selfadjoint operators that covers the equality $T=T^{\sharp_A}$; which holds if and only if $T$ is $A$-selfadjoint and $\mathcal{R}\left(T\right)\subseteq \overline{\mathcal{R}\left(A\right)}$. From now on, we call this property a $\sharp_A$-selfadjointness property.   In general, for $T\in \mathscr{B}_{A }\left(\mathscr{H}\right)$, neither  $T^{\sharp_A}T$ nor $TT^{\sharp_A}$ is positive. However, these operators are $A$-selfadjoint and $A$-positive.   
Thus, in viewing of these facts, we are able to define the $A$-absolute value operator  of $T$,  such as $\left|T\right|^2_A= AT^{\sharp_A}T$, which is positive operator, and we write $\left|T\right|_A=\left(AT^{\sharp_A}T\right)^{1/2}$.  This property is called the uniqueness of the square root of $A$-positive operators.
We note that,  the $A$-absolute value operator is selfadjoint if $T$ is $A$-selfadjoint and $A$ commutes with  $T$. Moreover, we have  
\begin{align*}
\left\| {T^{\sharp_A} T} \right\|_A  = \left\| {TT^{\sharp_A} } \right\|_A  = \left\| T \right\|_A^2  = \left\| {T^{\sharp_A} } \right\|_A^2. 
\end{align*} 

The following lemma plays a main role in the presentation of the mixed Schwarz inequality for semi-Hilbertian space oprtators.
 
\begin{lemma}
 \label{lemma3}Let $A\in \mathscr{B}^+\left(\mathscr{H}\right)$. If  $T\in \mathscr{B}_{A }\left(\mathscr{H}\right)$ is   $A$-positive  such that $AT =TA$, then $
\left[ {\begin{array}{*{20}c}
	\left|T\right|_A & {T^{\sharp_A } }  \\
	T & \left|{T^{\sharp_A } }\right|_A  \\
	\end{array}} \right]$ is   positive operator in $\mathscr{B}_{\bf{A}}\left(\mathscr{H}\oplus\mathscr{H}\right)$,  where $
{\bf{A}} = \left[ {\begin{array}{*{20}c}
A & 0  \\
0 & A  \\
\end{array}} \right]\in \mathscr{B}^+\left(\mathscr{H}\oplus\mathscr{H}\right)$.
\end{lemma}  
\begin{proof}
Let ${\bf{F}}=
\left[ {\begin{array}{*{20}c}
	0 & {T^{\sharp_A } }  \\
	T & 0  \\
	\end{array}} \right] \in \mathscr{B}_{{\bf{A}} }\left(\mathscr{H}\oplus\mathscr{H}\right)$. Since $T$ is $A$-positive and $AT^{\sharp_A} =T^{\sharp_A}A$,  therefore it is easy to see that ${\bf{F}}$ is  ${\bf{A}}$-positive,  and 
${\bf{F}}^{\sharp_{\bf{A}} }=\left[ {\begin{array}{*{20}c}
	0 & T^{\sharp_A}  \\
	{T^{\sharp_A } }^{\sharp_A } & 0  \\
	\end{array}} \right]$. Moreover, we have ${\bf{A}} {\bf{F}}^{\sharp_{\bf{A}}}={\bf{F}}^{\sharp_{\bf{A}}}{\bf{A}}$, 
and
\begin{align*}
 {\bf{A}}{\bf{F}}^{\sharp_{\bf{A}}} {\bf{F}} = \left[ {\begin{array}{*{20}c}
	{ AT^{\sharp_A } T} & 0  \\
	0 &  { A\left(T^{\sharp_A }\right)^{\sharp_A }T^{\sharp_A } }  \\
	\end{array}} \right]&= \left[ {\begin{array}{*{20}c}
	{ AT^{\sharp_A } T} & 0  \\
	0 &  {  \left(T^{\sharp_A }\right)^{\sharp_A }AT^{\sharp_A } }  \\
	\end{array}} \right]
\\
&= \left[ {\begin{array}{*{20}c}
	{ AT^{\sharp_A } T} & 0  \\
	0 &  { \left(T^{\sharp_A }\right)^{\sharp_A }T^{\sharp_A }A }  \\
	\end{array}} \right]
\\
&=  \left[ {\begin{array}{*{20}c}
	{\left| { T} \right|^2_A } & 0  \\
	0 & {\left| { T^{\sharp_A } } \right|^2_A }  \\
	\end{array}} \right]= \left| {\bf{F}} \right|^2_A
\ge0,
\end{align*}
Therefore, the uniqueness of the square root of $A$-positive operators, implies that 
\begin{align*} 
 \left| {\bf{F}} \right|_A =\left({\bf{A}}{\bf{F}}^{\sharp_{\bf{A}}} {\bf{F}} \right)^{1/2}=  \left[ {\begin{array}{*{20}c}
 	{\left| { T} \right|_A } & 0  \\
 	0 & {\left| { T^{\sharp_A } } \right|_A }  \\
 	\end{array}} \right] =\left[ {\begin{array}{*{20}c}
 	{ AT^{\sharp_A } T} & 0  \\
 	0 &  { \left(T^{\sharp_A }\right)^{\sharp_A }T^{\sharp_A }A }  \\
 	\end{array}} \right].
\end{align*}
Hence ${\bf{F}} + \left| {\bf{F}} \right|_A $ is  ${\bf{A}}$-positive; i.e.,
$\left[ {\begin{array}{*{20}c}
	\left|T\right|_A & {T^{\sharp_A } }  \\
	T & \left|{T^{\sharp_A } }\right|_A  \\
	\end{array}} \right]$ is $ {\bf{A}}$-positive operator in $\mathscr{B}_{ {\bf{A}} }\left(\mathscr{H}\oplus\mathscr{H}\right)$. 
\end{proof} 
 
Now, we are ready to state the corresponding new version of the mixed Schwarz inequality for semi-Hilbertian space oprtators.
  \begin{theorem}
	\label{thm1} 	Let $A \in \mathscr{B}^+\left(\mathscr{H}\right)$	 be any positive operator. If $T \in \mathscr{B}_{A} \left(\mathscr{H}\right)$ such that $AT=TA$.  
	If $f$ and $g$ are nonnegative continuous functions on $\left[0,\infty\right)$ satisfying $f(t)g(t) =t$ $(t\ge0)$, then
	\begin{align}
	\left| {\left\langle {T   x,y} \right\rangle_A } \right|   \le 
	\left\| {f\left( {\left| T \right|_A x} \right)} \right\|_A \left\|  {g\left( {\left| {T^{\sharp_A } } \right|_A y} \right)} \right\|_A.
	\label{eq2.2}
	\end{align}
	for all vectors $x,y\in \mathscr{H}$.
\end{theorem}
\begin{proof}
Since $T^{\sharp_A}=A^{\dag}T^*A$, then we observe that 
\begin{align*}
\left|T^{\sharp_A}\right|_A^2T = \left(T^{\sharp_A }\right)^{\sharp_A }T^{\sharp_A }AT &= \left(A^{\dag}T^*A\right)^{\sharp_A }A^{\dag}T^*AAT
\nonumber\\
&=A(T^*)^{\sharp_A}A^{\dag}A^{\dag}T^*AAT
\nonumber\\
&=A A^{\dag}TAA^{\dag}A^{\dag}T^*AAT
\nonumber\\
&=TT^*AT \qquad\qquad (\text{since}\,\, T^*A=AT^*)
\\
&=TAT^{\sharp_A}T
\\
&=T\left|T\right|_A^2,
\end{align*}
it follows that $T\left|T\right|_A=\left|T^{\sharp_A}\right|_AT$. Thus, by Lemmas \ref{lemma2} and \ref{lemma3}, we have $
\left[ {\begin{array}{*{20}c}	{f^2 \left( \left|T\right|_A \right)} & {T^{\sharp_A } }  \\
T & {g^2 \left( \left|T^{\sharp_A}\right|_A \right)}  \\
\end{array}} \right]$  is  ${{\bf{A}}}$-positive operator in $\mathscr{B}_{{\bf{A}}}\left(\mathscr{H}\oplus\mathscr{H}\right)$.
The required inequality now follows from Lemma \ref{lemma1}.
\end{proof}

\begin{remark}
Under the assumptions of Theorem \ref{thm1}. Choosing $f(t)=t^\alpha$ and $g(t)=t^{1-\alpha}$, we get	
\begin{align}
 \label{eq2.3}\left| {\left\langle {T   x,y} \right\rangle_A } \right|^2   \le \left\langle {\left| T \right|^{2\alpha}_A x,x} \right\rangle_A \left\langle {\left| {T^{\sharp_A } } \right|^{2\left(1-\alpha\right)}_A y,y} \right\rangle_A, \qquad  0\le \alpha \le 1
\end{align}
  for all vectors $x,y\in \mathscr{H}$. Setting $\alpha=\frac{1}{2}$, we get
\begin{align}
 \label{eq2.4}\left| {\left\langle {T   x,y} \right\rangle_A } \right|    \le \left\langle {\left| T \right|_A x,x} \right\rangle_A^{\frac{1}{2}} \left\langle {\left| {T^{\sharp_A } } \right|_A y,y} \right\rangle_A^{\frac{1}{2}}
\end{align}
  for all vectors $x,y\in \mathscr{H}$.
\end{remark}

  \begin{theorem}
	\label{thm2} 	Let $A \in \mathscr{B}^+\left(\mathscr{H}\right)$	 be any positive operator. If $T \in \mathscr{B}_{A^{1/2}} \left(\mathscr{H}\right)$ such that $A$ commutes with $T$ and $\left|T\right|S=S^{\sharp_A}\left|T\right|$.  
	If $f$ and $g$ are nonnegative continuous functions on $\left[0,\infty\right)$ satisfying $f(t)g(t) =t$ $(t\ge0)$, then
	\begin{align*}
	\left| {\left\langle {T   x,y} \right\rangle_A } \right|   \le  r_A\left(T\right)
	\left\| {f\left( {\left| T \right|_A x} \right)} \right\|_A \left\|  {g\left( {\left| {T^{\sharp_A } } \right|_A y} \right)} \right\|_A.
	\end{align*}
	for all vectors $x,y\in \mathscr{H}$.
\end{theorem}
\begin{proof}
The proof goes likewise the proof \cite[Theorem 5]{FK2} by rewriting the proof for the sem-inner produnct $ \left\langle {\cdot,\cdot} \right\rangle _A$, taking into account Theorem \ref{thm1}.
\end{proof} 

\begin{lemma}
\label{lemma4}	Let $f:$ be a non-negative convex function defined on a real interval $I$. Then for every  positive operator $T\in \mathscr{B}_{A}\left(\mathscr{H}\right)$ whose $\spe_A\left(T\right)\subset I$, we have
	\begin{align}
 \label{eq2.5}	f\left( { \left\langle {Tx,x} \right\rangle _A} \right) \le 
	\left\langle {f\left( T \right)x,x} \right\rangle _A 
	\end{align}
	for all  vector  $x \in  \mathscr{H} $. If $f$ is concave then the inequality is reversed.
\end{lemma}
\begin{proof}
	Since $f$ is convex then for any $x,t\in I$ there is a $\lambda \in \mathbb{R}$ such that
	\begin{align*}
	f\left( t \right) \ge f\left( s \right) + \lambda \left( {t - s} \right)
	\end{align*}
	Since $T$ is  positive, then $T$ is  selfadjoint operator. Using the functional calculus for sesquilinear form, thus   we have 
	\begin{align*}
	f\left( T \right) \ge f\left( s \right) + \lambda \left( {T - s} \right)
	\end{align*}
	which is equivalent to write
	\begin{align*}
	 \left\langle {f\left( T \right)x,x} \right\rangle _A 
	  \ge f\left( s \right) 1_{\mathscr{H}}+ \lambda   \left\langle { \left( {T - s1_{\mathscr{H}}} \right)x,x  } \right\rangle _A 
	\end{align*}
	for all vectors $x\in \mathscr{H}$. Setting $s=\left\langle {Tx,x} \right\rangle _A $, we have
	\begin{align*}
		\left\langle {f\left( T \right)x,x} \right\rangle _A 
	 &\ge	f\left( { \left\langle {Tx,x} \right\rangle _A} \right)  1_{\mathscr{H}}+ \lambda  	\left\langle  \left( {T -  \left\langle {Tx,x} \right\rangle _A  } \right)x,x \right\rangle _A 
 	\\
	&=f\left( { \left\langle {Tx,x} \right\rangle _A} \right)  + \lambda \left[\left\langle {Tx,x} \right\rangle _A -\left\langle {Tx,x} \right\rangle _A \right]
	\\
	&=f\left( { \left\langle {Tx,x} \right\rangle _A} \right)
	\end{align*}
	for all vectors $x\in \mathscr{H}$, and this proves the required result.
\end{proof}

The following version of H\"{o}lder--McCarty inequality holds  for semi-Hilbertian operators.
\begin{corollary} 
	Let $T\in \mathscr{B}_{A}\left(\mathscr{H}\right)$, such that $T$ is  positive  and $x\in \mathscr{H}$ be an $A$-unit vector. Then, 
	\begin{align}
	\label{eq2.6}  \left\langle {Tx,x} \right\rangle^r _A  \le  \left\langle {T^rx,x} \right\rangle _A,\qquad   r\ge1 
	\end{align}
	and
	\begin{align}
	\label{eq2.7} \left\langle {T^rx,x} \right\rangle _A \le   \left\langle {Tx,x} \right\rangle^r _A ,\qquad   0\le r\le1 
	\end{align}
\end{corollary}
\begin{proof}
	Let $f\left(t\right)=t^r$ $(r\ge1)$ in Lemma \ref{lemma4}. For the second inequality \eqref{eq2.6}, apply the reversed version of \eqref{eq2.5} for the function  $f\left(t\right)=t^r$ $(0\le r\le1)$.
\end{proof}
By noting that, for $A$-positive operator $T$ we have
\begin{align*}
\left\langle {Tx,x} \right\rangle^r _A = \left\langle {ATx,x} \right\rangle^r \le  \left\langle {(AT)^{r}x,x} \right\rangle =\left\langle {AT(AT)^{r-1}x,x} \right\rangle=\left\langle {T(AT)^{r-1}x,x} \right\rangle_A, \qquad \forall r\ge1
\end{align*} 
which implies that the inequality 
\begin{align*}
\left\langle {Tx,x} \right\rangle^r _A  \le  \left\langle {T(AT)^{r-1}x,x} \right\rangle _A,\qquad   r\ge1 
\end{align*}
holds if and only if $AT$ is positive, i.e., $T$ is $A$-positive; which indeed, the corresponding version of H\"{o}lder--McCarty inequality for $A$-positive operators act  on semi-Hilbertian spaces.
Similarly, we have
\begin{align*}
\left\langle {Tx,x} \right\rangle^r _A  \ge  \left\langle {T(AT)^{r-1}x,x} \right\rangle _A,\qquad   0\le r\le1 
\end{align*}
hold for $A$-positive operator $T$.
\begin{corollary} 
\label{cor2}	Let $T\in \mathscr{B}_{A }\left(\mathscr{H}\right)$, such that $T$ is $\sharp_A$-selfadjoint and $x\in \mathscr{H}$ be an $A$-unit vector. Then, 
	\begin{align}
	\label{eq2.8}  \left|\left\langle {Tx,x} \right\rangle_A\right|  \le  \left\langle {\left|T\right|_A x,x} \right\rangle _A 
	\end{align}

\end{corollary}
\begin{proof}
Since $T=T^{\sharp}$, by letting $y=x$ in \eqref{eq2.2} with $S=I$ we have the inequality
\eqref{eq2.8}.
\end{proof}

In fact, one may establish a generalization of Theorem \ref{thm1}  to several operators, by letting $T_i, \in
\mathscr{B}_A\left( \mathscr{H}\right)$ $(i=1,\cdots,n)$ such that
\begin{align*}
 AT_i  =  T_iA   \qquad {\rm{and}}   
\end{align*}
If $f,g$ are as above, proceeding as in the proof of Theorem \ref{thm1},
then we have
\begin{align}
 \left| {\left\langle {\left(\sum_{i=1}^{n} T_i \right)x, u} \right\rangle_A } \right|  
 &\le \sum_{i=1}^n{  \left\| {f\left( {\left| T_i \right|_A} \right)x} \right\|_A\left\| {g\left( {\left| {T_i^{\sharp_A} } \right|_A} \right)u} \right\|_A} \label{eq2.9}\\
&\le  \left(\sum_{i=1}^n \left\| {f\left( {\left| T_i
		\right|_A} \right)x} \right\|_A^p \right)^{1/p}  \left(\sum_{i=1}^n
\left\| {g\left( {\left| {T_i^{\sharp_A} } \right|_A} \right)u} \right\|_A^q
\right)^{1/q}.\nonumber
\end{align}
For all $x,u\in \mathscr{H}$, which follows by the H\"{o}lder inequality, where $p,q$ are conjugate exponents, i.e.,  $p,q>1$ with $\frac{1}{p}+\frac{1}{q}=1$.\\

Thus, one may has the following norm inequality
\begin{align}
\left\| \sum_{i=1}^{n} T_i \right\|_A \le  
\left(\sum_{i=1}^n \left\| {f\left( {\left| T_i \right|_A} \right) }
\right\|_A^p \right)^{1/p}  \left(\sum_{i=1}^n \left\| {g\left(
	{\left| {T_i^{\sharp_A} } \right|_A} \right)} \right\|_A^q
\right)^{1/q}.\label{eq2.10}
\end{align}
For instance,  consider $f(t)=t^{\alpha}$ and $g(t)=t^{1-\alpha}$,
one has from  \eqref{eq2.9} that
\begin{align*}
\left\| \sum_{i=1}^{n} T_i  \right\|_A  \le   
\left(\sum_{i=1}^n \left\| {\left| T_i \right|_A^{\alpha}  }
\right\|_A^p \right)^{1/p}  \left(\sum_{i=1}^n \left\| { \left|
	{T_i^{\sharp_A} } \right|_A^{1-\alpha} } \right\|_A^q \right)^{1/q}.
\end{align*}

\begin{remark}
In particular case  for $n=1$ (setting $T_1=S$), then we have
\begin{align*}
\left\|  S   \right\|_A  \le   
\left\| {\left| S \right|_A^{\alpha}  } \right\|_A 
  \left\| { \left| {S^{\sharp_A} } \right|_A^{1-\alpha} }
\right\|_A, \qquad 0\le \alpha\le 1.
\end{align*}
 Also, for $n=2$ and $p=q=2$ we get
\begin{align*}
\left\|  T_1+T_2  \right\|_A  \le  \left(\left\| {\left| T_1 \right|_A^{\alpha}  } \right\|_A^2+\left\| {\left| T_2 \right|_A^{\alpha}  } \right\|_A^2 \right)^{1/2}
\,\,  \left(\left\| {\left| T^{\sharp_A}_1 \right|_A^{1-\alpha}  } \right\|_A^2+\left\| {\left| T^{\sharp_A}_2 \right|_A^{1-\alpha}  } \right\|_A^2 \right)^{1/2}
\end{align*}
for all $\alpha \in \left[0,1\right]$. 
\end{remark}

The   next result provides a new extension of
the mixed Schwarz inequality \eqref{eq2.2} using $A$-Cartesian decomposition. 
\begin{theorem}
	\label{thm3}Let  $T\in \mathscr{B}_A\left( \mathscr{H}\right)$ such that $AT=TA$,
	with the $A$-Cartesian decomposition    $T=P+iQ$.  If $f$ and $g$ are as in
	Theorem \ref{thm1}, then
	\begin{align}
	\left| {\left\langle {T x,y} \right\rangle_A } \right| \le \left\| {f\left( \left|P\right|_A
		\right)x} \right\|_A\left\| {g\left( \left|P^{\sharp_A}\right|_A \right)y}
	\right\|_A + \left\| {f\left( \left|Q\right|_A \right)x}
	\right\|_A\left\| {g\left( \left|Q^{\sharp_A}\right|_A \right)y} \right\|_A  \label{eq2.11}
	\end{align}
	for all $x,y\in \mathscr{H}$.
\end{theorem}
\begin{proof}
	Let $ P+iQ$ be  the $A$-Cartesian decomposition of $T$. Then
	\begin{align*}
	\left| {\left\langle {Tx,y} \right\rangle_A } \right| &= \left( {\left\langle { Px,y} \right\rangle_A ^2  + \left\langle {Qx,y} \right\rangle_A ^2 } \right)^{1/2}  \\
	&\le \left| {\left\langle { P x,y} \right\rangle_A } \right| + \left| {\left\langle {Q x,y} \right\rangle_A } \right| 
	\\
	&\le  \left\| {f\left( \left|P\right|_A
			\right)x} \right\|_A\left\| {g\left( \left|P^{\sharp_A}\right|_A \right)y}
		\right\|_A + \left\| {f\left( \left|Q\right|_A \right)x}
		\right\|_A\left\| {g\left( \left|Q^{\sharp_A}\right|_A \right)y} \right\|_A 
	\end{align*}
	for all $x,y\in \mathscr{H}$, where the last inequality follows form \eqref{eq2.2}.  
\end{proof}
\begin{remark}
The above version of the mixed Schwarz inequality is a generalization of the main result in \cite{Alomari}.
\end{remark}

\begin{corollary}
	\label{cor3}    Let  $T\in \mathscr{B}\left( \mathscr{H}\right)$ such that $AT=TA$, with the $A$-Cartesian decomposition    $T=P+iQ$.  Then 
	\begin{align}
	\left| {\left\langle {Tx,y} \right\rangle_A } \right|
	\le  \left\{ {\left\| { \left|P\right|_A^{2\alpha}x} \right\|_A\left\| { \left|P^{\sharp_A} \right|_A^{2\left(1-\alpha\right)} y} \right\|_A + \left\| { \left|Q\right|_A^{2\alpha}x} \right\|_A\left\| { \left|Q^{\sharp_A}\right|_A^{2\left(1-\alpha\right)} y} \right\|_A} \right\} \label{eq2.12}
	\end{align}
	for all $x,y\in \mathscr{H}$.
\end{corollary}
\begin{proof}
	Setting $f(t)=t^{\alpha}$ and $g(t)=t^{1-\alpha}$, $0\le \alpha \le 1$, $t\ge0$ in Theorem \ref{thm2}  we get \eqref{eq2.10}.
\end{proof}

The $A$-Cartesian companion decomposition of the mixed Schwarz inequality
\eqref{eq2.3} can be deduced as follows:
\begin{corollary}
	\label{cor4}Let  $T  \in \mathscr{B}\left( \mathscr{H}\right)$ such that $AT=TA$,
	with the $A$-Cartesian decomposition    $A=P+iQ$.  Then 
	\begin{align}
	\left| {\left\langle {T x, y} \right\rangle_A } \right|  \le
 \frac{1}{2}  \left\langle {\left(\left| P \right|^{2\alpha }+\left| Q \right|^{2\alpha }\right) x,x} \right\rangle_A  +\frac{1}{2}  \left\langle {\left(\left| P^{\sharp_A} \right|^{2\left(1-\alpha\right)}+\left| Q^{\sharp_A} \right|^{2\left(1-\alpha\right)}\right) y,y} \right\rangle_A
	\end{align}
	for all $x,y\in \mathscr{H}$ and any $0\le \alpha\le 1$.
\end{corollary}
\begin{proof}
From \eqref{eq2.12} we have	
	\begin{align*}
\left| {\left\langle {T x, y} \right\rangle_A } \right|  &\le\left\langle {\left| P \right|^{2\alpha } x,x} \right\rangle_A^{1/2}
\left\langle {\left| P^{\sharp_A} \right|^{2\left( {1 - \alpha } \right)}	y,y} \right\rangle^{1/2}_A +\left\langle {\left| Q \right|_A^{2\alpha } x,x}
\right\rangle^{1/2}_A \left\langle {\left|Q ^{\sharp_A}\right|_A^{2\left( {1 - \alpha} \right)} y,y} \right\rangle^{1/2}_A
\\
&\le \frac{1}{2} \left[\left\langle {\left| P \right|^{2\alpha } x,x} \right\rangle_A+
\left\langle {\left| P^{\sharp_A} \right|^{2\left( {1 - \alpha } \right)}	y,y} \right\rangle_A \right]+\frac{1}{2} \left[\left\langle {\left| Q \right|^{2\alpha } x,x} \right\rangle_A+
\left\langle {\left| Q^{\sharp_A} \right|^{2\left( {1 - \alpha } \right)}	y,y} \right\rangle_A \right]
\\
&\le \frac{1}{2} \left[\left\langle {\left| P \right|^{2\alpha } x,x} \right\rangle_A+
\left\langle {\left| Q \right|^{2\alpha } x,x} \right\rangle_A \right]+\frac{1}{2} \left[ \left\langle {\left| P^{\sharp_A} \right|^{2\left( {1 - \alpha } \right)}	y,y} \right\rangle_A  +
\left\langle {\left| Q^{\sharp_A} \right|^{2\left( {1 - \alpha } \right)}	y,y} \right\rangle_A \right]
\\
&\le \frac{1}{2}  \left\langle {\left(\left| P \right|^{2\alpha }+\left| Q \right|^{2\alpha }\right) x,x} \right\rangle_A  +\frac{1}{2}  \left\langle {\left(\left| P^{\sharp_A} \right|^{2\left(1-\alpha\right)}+\left| Q^{\sharp_A} \right|^{2\left(1-\alpha\right)}\right) y,y} \right\rangle_A,
\end{align*}
which gives the required result.
\end{proof}

\section{  $A$-numerical radius inequalities}
 In this section some  inequalities for the $A$-numerical radius are presented, indeed the next two results generalizes the first two results in \cite{EF}.
\begin{theorem}
	Let    $T\in \mathscr{B}_{A} \left(\mathscr{H}\right)$, such that $AT=TA$, $0 \le \alpha \le 1$ and $r\ge1$. Then
	\begin{align}
	w^r_A\left(T\right) \le 
	\frac{1}{2}   \left\| {   \left| T \right|^{2r\alpha}_A+\left| {T^{\sharp_A } } \right|^{2r\left(1-\alpha\right)}_A  } \right\|_A
	\end{align}
\end{theorem}
\begin{proof}
	Let $x\in \mathscr{H}$ be $A$-unit vector, then 
	\begin{align*}
	\left| {\left\langle {T   x,x} \right\rangle_A } \right|    &\le \left\langle {\left| T \right|^{2\alpha}_A x,x} \right\rangle_A^{1/2} \left\langle {\left| {T^{\sharp_A } } \right|^{2\left(1-\alpha\right)}_A x,x} \right\rangle_A^{1/2}\qquad\qquad \text{(by \eqref{eq2.3})}
	\\
	&\le \left({ \frac{\left\langle {\left| T \right|^{2\alpha}_A x,x} \right\rangle_A^{r} + \left\langle {\left| {T^{\sharp_A } } \right|^{2\left(1-\alpha\right)}_A x,x} \right\rangle^{r}_A}{2} }\right)^{\frac{1}{r}}  \qquad  \text{(by Power mean inequality)}
	\\
	&\le \left({ \frac{\left\langle {\left| T \right|^{2r\alpha}_A x,x} \right\rangle_A + \left\langle {\left| {T^{\sharp_A } } \right|^{2r\left(1-\alpha\right)}_A x,x} \right\rangle_A}{2} }\right)^{\frac{1}{r}}\qquad  \text{(by \eqref{eq2.6})} 
	\end{align*}
	Therefore,
	\begin{align*}
	\left| {\left\langle {T   x,x} \right\rangle_A } \right| ^r \le \frac{1}{2}   \left\langle { \left( \left| T \right|^{2r\alpha}_A+\left| {T^{\sharp_A } } \right|^{2r\left(1-\alpha\right)}_A  \right)x,x} \right\rangle_A.     
	\end{align*}
	Taking the supremum over all $A$-unit vector $x\in \mathscr{H}$ we get the required result.   
\end{proof}

\begin{theorem}
		Let    $T\in \mathscr{B}_{A} \left(\mathscr{H}\right)$, such that $AT=TA$, $0 \le \alpha \le 1$ and $r\ge1$. Then
	\begin{align}
	w^{2r}_A\left(T\right) \le 
	\left\| {   \alpha\left| T \right|^{2r}_A +\left(1-\alpha\right)\left| {T^{\sharp_A } } \right|^{2r}_A    } \right\|_A
	\end{align}
\end{theorem}
\begin{proof}
	Let $x\in \mathscr{H}$ be $A$-unit vector, then 
	\begin{align*}
	\left| {\left\langle {T   x,x} \right\rangle_A } \right|^2    &\le \left\langle {\left| T \right|^{2\alpha}_A x,x} \right\rangle_A  \left\langle {\left| {T^{\sharp_A } } \right|^{2\left(1-\alpha\right)}_A x,x} \right\rangle_A  \qquad\qquad\qquad\qquad \text{(by \eqref{eq2.3})}
	\\
	&\le \left\langle {\left| T \right|^{2}_A x,x} \right\rangle^{\alpha}_A  \left\langle {\left| {T^{\sharp_A } } \right|^{2}_A x,x} \right\rangle^{\left(1-\alpha\right)}_A\qquad \qquad\qquad\qquad\text{(by \eqref{eq2.7})}
	\\
	&\le  \left(\alpha\left\langle {\left| T \right|^{2}_A x,x} \right\rangle^r_A  +\left(1-\alpha\right)\left\langle {\left| {T^{\sharp_A } } \right|^{2}_A x,x} \right\rangle^r_A \right)^{1/r}\qquad \text{(by AM-GM inequality)}
	\\
	&\le  \left(\alpha\left\langle {\left| T \right|^{2r}_A x,x} \right\rangle_A  +\left(1-\alpha\right)\left\langle {\left| {T^{\sharp_A } } \right|^{2r}_A x,x} \right\rangle_A \right)^{1/r}\qquad \text{(by \eqref{eq2.6})}
	\\
	&\le   \left\langle {\left(\alpha\left| T \right|^{2r}_A +\left(1-\alpha\right)\left| {T^{\sharp_A } } \right|^{2r}_A \right) x,x} \right\rangle_A ^{1/r}.
	\end{align*}
	Therefore,
	\begin{align*}
	\left| {\left\langle {T   x,x} \right\rangle_A } \right|^{2r} \le \left\langle {\left(\alpha\left| T \right|^{2r}_A +\left(1-\alpha\right)\left| {T^{\sharp_A } } \right|^{2r}_A \right) x,x} \right\rangle_A.     
	\end{align*}
	Taking the supremum over all $A$-unit vector $x\in \mathscr{H}$ we get the required result.   
\end{proof}

\begin{theorem}
	\label{thm6}	Let    $T\in \mathscr{B}_{A} \left(\mathscr{H}\right)$, such that $AT=TA$, with the $A$-Cartesian decomposition    $T=P+iQ$.   If $f$ and $g$ are as in Theorem \ref{thm1}. Then
	\begin{align}
	w_A \left(T\right)  &\le   \left\| {f^p \left(
		{\left| P \right|_A} \right) + f^p \left( {\left| Q \right|}
		\right)} \right\|_A^{1/p} \left\| {g^q \left( {\left| P^{\sharp_A} \right|_A}
		\right) + g^q \left( {\left| Q^{\sharp_A} \right|_A} \right)} \right\|_A^{1/q}
	\\
	&\le \left\|\frac{1}{p}\left[ {f^p \left( {\left| P \right|_A} \right) + f^p \left( {\left| Q \right|_A} \right)} \right]+\frac{1}{q}\left[ {g^q \left( {\left| P^{\sharp_A} \right|_A} \right) + g^q \left( {\left| Q^{\sharp_A} \right|_A} \right)} \right]\right\|\nonumber
	\end{align}
	for all $p,q\ge2$ with $\frac{1}{p}+\frac{1}{q}=1$.
\end{theorem}

\begin{proof}
	Letting $y=x$ in  \eqref{eq2.11}, then we have
	\begin{align*}
	&\left| {\left\langle {Tx,y} \right\rangle_A } \right|   \\
	&\le     \left\{ {\left\| {f\left( {\left| P \right|_A} \right)x} \right\|_A\left\| {g\left( {\left| P^{\sharp_A} \right|} \right)y} \right\|_A + \left\| {f\left( {\left| Q \right|_A} \right)x} \right\|_A\left\| {g\left( {\left| Q^{\sharp_A} \right|_A} \right)y} \right\|_A} \right\}
	\\
	&\le   
	\left( {\left\| {f\left( {\left| P \right|_A} \right)x} \right\|_A^p  + \left\| {f\left( {\left| Q \right|_A} \right)x} \right\|_A^p } \right)^{1/p}
	\\
	&\qquad\qquad\times\left( {\left\| {g\left( {\left| P^{\sharp_A} \right|_A} \right)y} \right\|_A^q  + \left\| {g\left( {\left| Q^{\sharp_A} \right|_A} \right)y} \right\|_A^q } \right)^{1/q} \nonumber \qquad\qquad \qquad {(\rm{by\,\,H\text{\"{o}}lder\,\, inequaity})}
	\\
	&=\left( {\left\langle {f^2 \left( {\left| P \right|_A} \right)x,x} \right\rangle_A^{p/2}  + \left\langle {f^2 \left( {\left| Q \right|_A} \right)x,x} \right\rangle_A ^{p/2} } \right)^{1/p}
	\\
	&\qquad \times\left( {\left\langle {g^2 \left( {\left| P^{\sharp_A} \right|} \right)x,x} \right\rangle_A ^{q/2}  + \left\langle {g^2 \left( {\left| Q^{\sharp_A} \right|_A} \right)x,x} \right\rangle_A ^{q/2} } \right)^{1/q}  \nonumber
	\\
	&\le \left( {\left\langle {f^p \left( {\left| P \right|_A} \right)x,x} \right\rangle_A  + \left\langle {f^p \left( {\left| Q_A \right|} \right)x,x} \right\rangle_A } \right)^{1/p}
	\\
	&\qquad\times \left( {\left\langle {g^q \left( {\left| P^{\sharp_A} \right|_A} \right)x,x} \right\rangle  + \left\langle {g^q \left( {\left| Q^{\sharp_A} \right|_A} \right)x,x} \right\rangle_A } \right)^{1/q}\qquad\qquad\qquad   \qquad\qquad{\text({\rm{by}}\,\, \eqref{eq2.6})}\nonumber
	\\
	&\le \left\langle {\left[ {f^p \left( {\left| P \right|_A} \right) + f^p \left( {\left| Q \right|_A} \right)} \right]x,x} \right\rangle_A ^{1/p} \left\langle {\left[ {g^q \left( {\left| P^{\sharp_A} \right|_A} \right) + g^q \left( {\left| Q^{\sharp_A} \right|_A} \right)} \right]x,x} \right\rangle_A ^{1/q}  \nonumber
	\\
	&\le \frac{1}{p}\left\langle {\left[ {f^p \left( {\left| P \right|_A} \right) + f^p \left( {\left| Q \right|_A} \right)} \right]x,x} \right\rangle_A + \frac{1}{q}\left\langle {\left[ {g^q \left( {\left| P^{\sharp_A} \right|_A} \right) + g^q \left( {\left| Q^{\sharp_A} \right|_A} \right)} \right]x,x} \right\rangle_A  \qquad  \text{(by AM-GM inequality)}
	\\
	&\le \left\langle {\left\{\frac{1}{p}\left[ {f^p \left( {\left| P \right|_A} \right) + f^p \left( {\left| Q \right|_A} \right)} \right]+\frac{1}{q}\left[ {g^q \left( {\left| P^{\sharp_A} \right|_A} \right) + g^q \left( {\left| Q^{\sharp_A} \right|_A} \right)} \right]\right\} x,x} \right\rangle_A  	   
	\end{align*}
	for all $p,q\ge2$ with $\frac{1}{p}+\frac{1}{q}=1$. Taking the
	supremum over all $A$-unit vector   $x\in \mathscr{H}$ we get the
	desired result.
	
\end{proof}

\begin{corollary}
	Let    $T\in \mathscr{B}_{A} \left(\mathscr{H}\right)$, such that $AT=TA$,
	with the $A$-Cartesian decomposition    $T=P+iQ$.   If $f$ and $g$ are as in
	Theorem \ref{thm1}. Then
	\begin{align*}
	w_A \left(T\right)  &\le   	\sqrt {\left\| {\left| P \right|_A^{2\alpha }  + \left| Q \right|_A^{2\alpha } } \right\|} \,\,\,\sqrt {\left\| {\left| {P^{\sharp_A} } \right|_A^{2\left( {1 - \alpha } \right)}  + \left| {Q^{\sharp_A} } \right|_A^{2\left( {1 - \alpha } \right)} } \right\|} 
	\\
	&\le 	 \frac{1}{2} \left\| {\left| P \right|_A^{2\alpha }  + \left| Q \right|_A^{2\alpha }  +  \left| {P^{\sharp_A} } \right|_A^{2\left( {1 - \alpha } \right)}  + \left| {Q^{\sharp_A} } \right|_A^{2\left( {1 - \alpha } \right)} } \right\|   
	\end{align*}
	for $0\le \alpha\le 1$.
\end{corollary}
\begin{proof}
Take $f(t)=t^{\alpha}$ and $g(t)=t^{1-\alpha}$ $(0\le \alpha\le 1)$, and setting $p=q=2$ in Theorem \ref{thm6}.
\end{proof}
\begin{theorem}
\label{thm7}Let $T_i, \in
\mathscr{B}_A\left( \mathscr{H}\right)$ $(i=1,\cdots,n)$ such that
$AT_i = T_iA$. Then,	
\begin{align}
w^p_A\left(\sum_{i=1}^{n} T_i \right) \le
\frac{1}{2n^{p-1}} \left\| \sum_{i=1}^n{\left(\left| {T_i } \right|_A^{2p\alpha } +\left| {T^{\sharp_A}_i } \right|_A^{2p\left( {1 - \alpha } \right)}\right)} \right\|_A .
\end{align}
\end{theorem}
for all $p\ge1$ and all $0\le \alpha \le 1$.
\begin{proof}
Setting $f(t)=t^{\alpha}$ and $g(t)=t^{1-\alpha}$ $(0\le \alpha\le 1)$ in the first inequality \eqref{eq2.9}, we have	
\begin{align*}
\left| {\left\langle {\left(\sum_{i=1}^{n} T_i \right)x, x} \right\rangle_A } \right| \nonumber 
&\le \sum_{i=1}^n{  \left\| { \left| T_i \right|^{2\alpha}_A x} \right\|_A\left\| { \left| {T_i^{\sharp_A} } \right|^{2\left(1-\alpha\right)}_A x} \right\|_A}  
\\
&\le   \sum_{i=1}^n{	\left\langle {\left| {T_i } \right|_A^{2\alpha } x,x} \right\rangle_A ^{\frac{1}{2}} \left\langle {\left| {T_i } \right|_A^{2\left( {1 - \alpha } \right)} x, x} \right\rangle_A ^{\frac{1}{2}}}. 
\end{align*}
It follows that
\begin{align*}
 \left| {\left\langle {\left(\sum_{i=1}^{n} T_i \right)x, x} \right\rangle_A } \right|^p \nonumber 
&\le   \left(\sum_{i=1}^n{	\left\langle {\left| {T_i } \right|_A^{2\alpha } x,x} \right\rangle_A ^{\frac{1}{2}} \left\langle {\left| {T^{\sharp_A}_i } \right|_A^{2\left( {1 - \alpha } \right)} x,x} \right\rangle_A ^{\frac{1}{2}}} \right)^p
\\
&\le \frac{1}{n^{p-1}} \sum_{i=1}^n{	\left\langle {\left| {T_i } \right|_A^{2\alpha } x,x} \right\rangle_A ^{\frac{p}{2}} \left\langle {\left| {T^{\sharp_A}_i } \right|_A^{2\left( {1 - \alpha } \right)} x,x} \right\rangle_A ^{\frac{p}{2}}}  
\\
&\le \frac{1}{n^{p-1}}  \sum_{i=1}^n{	\left\langle {\left| {T_i } \right|_A^{2p\alpha } x,x} \right\rangle_A ^{\frac{1}{2}} \left\langle {\left| {T^{\sharp_A}_i } \right|_A^{2p\left( {1 - \alpha } \right)} x,x} \right\rangle_A ^{\frac{1}{2}}} \qquad \text{(by \eqref{eq2.6})}
\\
&\le \frac{1}{n^{p-1}}   \sum_{i=1}^n{	\frac{\left\langle {\left| {T_i } \right|_A^{2p\alpha } x,x} \right\rangle_A + \left\langle {\left| {T^{\sharp_A}_i } \right|_A^{2p\left( {1 - \alpha } \right)} x,x} \right\rangle_A}{2} }  \qquad \text{(by AM-GM inequality)}
\\
&=\frac{1}{2n^{p-1}}   \sum_{i=1}^n{	 \left\langle {\left| {T_i } \right|_A^{2p\alpha } +\left| {T^{\sharp_A}_i } \right|_A^{2p\left( {1 - \alpha } \right)} x,x} \right\rangle_A  }  
\\
&=\frac{1}{2n^{p-1}} 	 \left\langle {\sum_{i=1}^n{\left(\left| {T_i } \right|_A^{2p\alpha } +\left| {T^{\sharp_A}_i } \right|_A^{2p\left( {1 - \alpha } \right)}\right)} x,x} \right\rangle_A    
\end{align*}
for all $p\ge1$, which proves the required result. 
\end{proof} 
\begin{corollary}
	Let    $T_1,T_2\in \mathscr{B}_{A} \left(\mathscr{H}\right)$, such that $AT_i=T_iA$ $(i=1,2)$. Then,	
	\begin{align}
	w^p_A\left( T_1 + T_2 \right)  \le
	\frac{1}{2^{p}}  \left\| \left| {T_1 } \right|_A^{2p\alpha } +\left| {T^{\sharp_A}_1 } \right|_A^{2p\left( {1 - \alpha } \right)} +\left| {T_2 } \right|_A^{2p\alpha } +\left| {T^{\sharp_A}_2 } \right|_A^{2p\left( {1 - \alpha } \right)} \right\|_A  
	\end{align}
for $0\le \alpha \le 1$.	
\end{corollary}
\begin{proof}
Setting $n=2$ in Theorem \ref{thm7}.
\end{proof}

\begin{theorem}
Let $B,T,C,E,S,F\in\mathscr{B}_A\left( \mathscr{H}\right)$ such that $A$ commutes with both $T$ and $S$. Then, 
\begin{align}
w_A\left(CTD+ESF\right) \le \frac{1}{2}\left\|D^{\sharp_A}  \left| T \right|_A^{2\alpha } D + C\left| {T^{\sharp_A}  } \right|_A^{2\left( {1 - \alpha } \right)} C^{\sharp_A}   + F^{\sharp_A}  \left| S \right|_A^{2\alpha } F + E\left| {S^{\sharp_A}  } \right|_A^{2\left( {1 - \alpha } \right)} E^{\sharp_A} \right\|_A \label{eq3.6}
\end{align} 
for $0\le \alpha \le 1$.
\end{theorem}
\begin{proof}
Employing the trianlge inequality, the mixed Schwarz inequality \eqref{eq2.3}, and then the AM-GM inequality, it follows that  
\begin{align*}
&\left| {\left\langle {\left( {CTD + ESF} \right)x,x} \right\rangle _A } \right| 
\\
&\le \left| {\left\langle {CTDx,x} \right\rangle _A } \right| + \left| {\left\langle {ESFx,x} \right\rangle _A } \right| 
\\ 
&= \left| {\left\langle {TDx,C^{\sharp_A}  x} \right\rangle _A } \right| + \left| {\left\langle {SFx,E^{\sharp_A}  x} \right\rangle _A } \right| 
\\ 
&\le \left\langle {\left| T \right|_A^{2\alpha } Dx,Dx} \right\rangle _A^{\frac{1}{2}} \left\langle {\left| {T^{\sharp_A}  } \right|_A^{2\left( {1 - \alpha } \right)} C^{\sharp_A}  x,C^{\sharp_A}  x} \right\rangle _A^{\frac{1}{2}}  
\\
&\qquad+ \left\langle {\left| S \right|_A^{2\alpha } Fx,Fx} \right\rangle _A^{\frac{1}{2}} \left\langle {\left| {S^{\sharp_A}  } \right|_A^{2\left( {1 - \alpha } \right)} E^{\sharp_A}  x,E^{\sharp_A}  x} \right\rangle _A^{\frac{1}{2}}  
\\ 
&\le \frac{1}{2}\left[ {\left\langle {\left| T \right|_A^{2\alpha } Dx,Dx} \right\rangle _A  + \left\langle {\left| {T^{\sharp_A}  } \right|_A^{2\left( {1 - \alpha } \right)} C^{\sharp_A}  x,C^{\sharp_A}  x} \right\rangle _A  }\right.
\\
&\qquad\qquad\left.{+ \left\langle {\left| S \right|_A^{2\alpha } Fx,Fx} \right\rangle _A  + \left\langle {\left| {S^{\sharp_A}  } \right|_A^{2\left( {1 - \alpha } \right)} E^{\sharp_A}  x,E^{\sharp_A}  x} \right\rangle _A } \right] 
\\ 
&= \frac{1}{2}\left[ {\left\langle {D^{\sharp_A}  \left| T \right|_A^{2\alpha } Dx,x} \right\rangle _A  + \left\langle {C\left| {T^{\sharp_A}  } \right|_A^{2\left( {1 - \alpha } \right)} C^{\sharp_A}  x,x} \right\rangle _A  }\right.
\\
&\qquad\qquad\left.{+ \left\langle {F^{\sharp_A}  \left| S \right|_A^{2\alpha } Fx,x} \right\rangle _A  + \left\langle {E\left| {S^{\sharp_A}  } \right|_A^{2\left( {1 - \alpha } \right)} E^{\sharp_A}  x,x} \right\rangle _A } \right] 
\\ 
&= \frac{1}{2}\left\langle {\left( {D^{\sharp_A}  \left| T \right|_A^{2\alpha } D + C\left| {T^{\sharp_A}  } \right|_A^{2\left( {1 - \alpha } \right)} C^{\sharp_A}   + F^{\sharp_A}  \left| S \right|_A^{2\alpha } F + E\left| {S^{\sharp_A}  } \right|_A^{2\left( {1 - \alpha } \right)} E^{\sharp_A}  } \right)x,x} \right\rangle _A.  
\end{align*}
Taking the
supremum over all $A$-unit vector   $x\in \mathscr{H}$ we get the
desired result.
\end{proof}
Inequality \eqref{eq3.6} yields several numerical radius inequalities as special
cases. A sample of elementary inequalities is demonstrated in the following
remarks.
\begin{remark}
Letting $T = I$ (the identity operator) and $S = 0$ in Theorem
2, it follows that $\left| I \right|^2_A=\left| T \right|^2_A=AT^{\sharp_A} T=AI=A$, we obtain the inequality
\begin{align}
w_A\left(CD\right) \le \frac{1}{2}\left\|D^{\sharp_A}  A^{\alpha } D + CA^{1 - \alpha } C^{\sharp_A}    \right\|_A.
\end{align} 
\end{remark}
\begin{remark}
Letting $C=D=E=F = I$ in Theorem 2,   we obtain the inequality
\begin{align}
w_A\left(T+S\right) \le \frac{1}{2}\left\|   \left| T \right|_A^{2\alpha }   +  \left| {T^{\sharp_A}  } \right|_A^{2\left( {1 - \alpha } \right)}    +    \left| S \right|_A^{2\alpha }  +  \left| {S^{\sharp_A}  } \right|_A^{2\left( {1 - \alpha } \right)}   \right\|_A.
\end{align}  
\end{remark}

\begin{remark}
Letting $T = S = I$, $E = D$, and $ F= \pm C$ in Theorem 2, we
obtain the inequality	
\begin{align}
w_A\left(CD\pm DC\right) \le \frac{1}{2}\left\| C^{\sharp_A}   A^{ \alpha } C +C A^{ 1 - \alpha  } C^{\sharp_A}   + D^{\sharp_A}  A^{ \alpha } D+ D A^{ 1 - \alpha } D^{\sharp_A} \right\|_A,
\end{align} 
which gives an estimate for the numerical radius of the commutator	$CD\pm DC$.
\end{remark}

\section{Generalized Euclidean $A$-numerical radius inequalities}

In 2018,  Baklouti \etal \cite{BFA} introduced the concept of Euclidean operator $A$-radius of an $n$-tuple $ {\bf{T}} = \left( {T_1 , \cdots ,T_n } \right) \in  \mathscr{B}\left(\mathscr{H}\right)^{n}:=\mathscr{B}\left(\mathscr{H}\right)\times \cdots \times \mathscr{B}\left(\mathscr{H}\right)$. Namely, for $T_1,\cdots,T_n \in \mathscr{B}\left(\mathscr{H}\right)$.
\begin{align*} 
w_{\rm{e},A} \left( {T_1 , \cdots ,T_n } \right): = \mathop {\sup }\limits_{\left\| x \right\| = 1} \left( {\sum\limits_{i = 1}^n {\left| {\left\langle {T_i x,x} \right\rangle_A } \right|^2 } } \right)^{1/2}.
\end{align*}
In the same work, the authors proved that
\begin{align}
\label{eq4.1}\frac{1}{{2\sqrt n }}\left\| {\sum\limits_{k = 1}^n {T_k T_k^{\sharp_A} } } \right\|^{\frac{1}{2}}  \le w_{A} \left( {T_1 , \cdots ,T_n } \right) \le \left\| {\sum\limits_{k = 1}^n {T_k T_k^{\sharp_A} } } \right\|^{\frac{1}{2}}. 
\end{align}
As a direct consequence of \eqref{eq4.1}; if $T = B + iC$ is the $A$-Cartesian decomposition of $A$, then
\begin{align*}
w_{\rm{e},A}^2 \left( {B,C} \right) = \mathop {\sup }\limits_{\left\| x \right\|_A = 1} \left\{ {\left| {\left\langle {Bx,x} \right\rangle_A } \right|^2  + \left| {\left\langle {Cx,x} \right\rangle_A } \right|^2 } \right\} = \mathop {\sup }\limits_{\left\| x \right\|_A = 1} \left| {\left\langle {Tx,x} \right\rangle_A } \right|^2  = w_A^2 \left( T \right).
\end{align*}
But since $T^{\sharp_A} T + TT^{\sharp_A}  = 2\left( {B^2  + C^2 } \right)$, then we obtain
\begin{align}
\frac{1}{16}\|T^{\sharp_A}T+TT^{\sharp_A}\|_A\le  w_{A}^2\left(T\right) \le \frac{1}{2}\|T^{\sharp_A}T+TT^{\sharp_A}\|_A.   \label{eq4.2}
\end{align}
 This inequality were improved by Bhunia \etal in \cite{BKP2} for $T\in \mathscr{B}_{A^{1/2}}\left(\mathscr{H}\right)$ and independently generalized by Feki in \cite{Kais2} for $T\in\mathscr{B}_A\left(\mathscr{H}\right)$, where they proved that
 \begin{align}
 \frac{1}{4} \left\|TT^{\sharp_A}+T^{\sharp_A}T\right\|_A\le w^2_A\left(T\right)  \le  \frac{1}{2} \left\|TT^{\sharp_A}+T^{\sharp_A}T\right\|_A. \label{eq4.3}
 \end{align}
The sharpness of \eqref{eq4.3} can be found in \cite{Kais2}. It should be noted that the inequality \eqref{eq4.3} is a kind generaliztion of Kittaneh inequality \cite{FK1} 
 \begin{align*}
 \frac{1}{4}\|T^*T+TT^*\|\le  w^2\left(T\right) \le \frac{1}{2}\|T^*T+TT^*\|  
 \end{align*}
 for Hilbert space operator $T\in \mathscr{B}\left(\mathscr{H}\right)$. These inequalities are sharp. 
 
 Let $T=B+iC$ be the $A$-Cartesian decomposition, then $B$ and $C$ are $A$-selfadjoint and $T^{\sharp_A}T+TT^{\sharp_A}=2 \left(B^2+C^2\right)$ . Therefore, \eqref{eq4.3} can be reformulated as 
 \begin{align*}
 \frac{1}{2} \left\|B^2+C^2\right\|_A \le w^2_A\left(T\right) \le \left\|B^2+C^2\right\|_A.
 \end{align*}
 or equivalently as
  \begin{align}
 \frac{1}{4} \left\|\left(B+C\right)^2+\left(B-C\right)^2\right\|_A \le w^2_A\left(T\right) \le  \frac{1}{2} \left\|\left(B+C\right)^2+\left(B-C\right)^2\right\|_A. \label{eq4.4}
 \end{align}
 
The purpose of this section is to generalize the Euclidean operator $A$-radius for $n$-tuple $n$-tuple $ {\bf{T}} = \left( {T_1 , \cdots ,T_n } \right) \in  \mathscr{B}\left(\mathscr{H}\right)^{n}:=\mathscr{B}\left(\mathscr{H}\right)\times \cdots \times \mathscr{B}\left(\mathscr{H}\right)$, which  extends \eqref{eq4.3} and also improves \eqref{eq4.1}. In lighting of the above $A$-Cartesian decomposition of the inequality \eqref{eq4.2}, a generalization of the inequality \eqref{eq4.4} is given as well.\\

We may start this section with the following result.
\begin{theorem}
Let $T\in \mathscr{B}_A\left(\mathscr{H}\right)$ such that $AT=TA$, with the $A$-Cartesian decomposition $T =
B + iC$, $0\le \alpha\le 1$, and $r\ge1$. Then	
\begin{align}
w^r_A \left(T\right)\le \frac{1}{2}\left\|\left| B \right|^{2r\alpha}_A +\left| {B^{\sharp_A } } \right|^{2r\left(1-\alpha\right)}_A+\left| C \right|^{2r\alpha}_A +\left| {C^{\sharp_A } } \right|^{2r\left(1-\alpha\right)}_A\right\|_A. 
\end{align}
\end{theorem}
\begin{proof}
	Since we have
\begin{align*}
\left| {\left\langle {T x,x} \right\rangle_A } \right|  &= \left( {\left\langle {B x,x} \right\rangle_A ^2  + \left\langle {C x,x} \right\rangle_A ^2 } \right)^{1/2}  
\\ 
&= \left( {\left|\left\langle {B x,x} \right\rangle_A\right| ^r  +\left| \left\langle {C x,x} \right\rangle_A\right| ^r } \right)^{\frac{1}{r}}
\\
&\le \left(\left\langle {\left| B \right|^{2\alpha}_A x,x} \right\rangle^{\frac{r}{2}}_A \left\langle {\left| {B^{\sharp_A } } \right|^{2\left(1-\alpha\right)}_A x,x} \right\rangle^{\frac{r}{2}}_A+\left\langle {\left| C \right|^{2\alpha}_A x,x} \right\rangle^{\frac{r}{2}}_A \left\langle {\left| {C^{\sharp_A } } \right|^{2\left(1-\alpha\right)}_A x,x} \right\rangle^{\frac{r}{2}}_A  \right)^\frac{1}{r} 
\\
&\le \frac{1}{2^{\frac{1}{r}}}\left(\left\langle {\left| B \right|^{2\alpha}_A x,x} \right\rangle^{r}_A +\left\langle {\left| {B^{\sharp_A } } \right|^{2\left(1-\alpha\right)}_A x,x} \right\rangle^{r}_A+\left\langle {\left| C \right|^{2\alpha}_A x,x} \right\rangle^{r}_A +\left\langle {\left| {C^{\sharp_A } } \right|^{2\left(1-\alpha\right)}_A x,x} \right\rangle^{r}_A \right)^{\frac{1}{r}}
\\
&\le \frac{1}{2^{\frac{1}{r}}}\left(\left\langle {\left| B \right|^{2r\alpha}_A x,x} \right\rangle_A +\left\langle {\left| {B^{\sharp_A } } \right|^{2r\left(1-\alpha\right)}_A x,x} \right\rangle_A+\left\langle {\left| C \right|^{2r\alpha}_A x,x} \right\rangle_A +\left\langle {\left| {C^{\sharp_A } } \right|^{2r\left(1-\alpha\right)}_A x,x} \right\rangle_A \right)^{\frac{1}{r}}
\\
&= \frac{1}{2^{\frac{1}{r}}}\left\langle {\left[\left| B \right|^{2r\alpha}_A +\left| {B^{\sharp_A } } \right|^{2r\left(1-\alpha\right)}_A+\left| C \right|^{2r\alpha}_A +\left| {C^{\sharp_A } } \right|^{2r\left(1-\alpha\right)}_A\right]x,x} \right\rangle^{\frac{1}{r}}_A, 
\end{align*}
which implies that
\begin{align*}
\left| {\left\langle {T x,x} \right\rangle_A } \right|^r\le \frac{1}{2}\left\langle {\left[\left| B \right|^{2r\alpha}_A +\left| {B^{\sharp_A } } \right|^{2r\left(1-\alpha\right)}_A+\left| C \right|^{2r\alpha}_A +\left| {C^{\sharp_A } } \right|^{2r\left(1-\alpha\right)}_A\right]x,x} \right\rangle_A.
\end{align*}
Taking the supremum over all $A$-unit vector $x\in \mathscr{H}$, we get the required result. 
\end{proof}

The generalized Euclidean operator $A$-radius of $T_1,\cdots,T_n $  would be defined as
\begin{align*} 
w_{p,A} \left( {T_1 , \cdots ,T_n } \right): = \mathop {\sup }\limits_{\left\| x \right\| = 1} \left( {\sum\limits_{i = 1}^n {\left| {\left\langle {T_i x,x} \right\rangle_A } \right|^p } } \right)^{1/p}, \qquad  p\ge1.
\end{align*}
This generalizes  the concepts of Euclidean operator radius of an $n$-tuple considered by Baklouti \etal \cite{BFA}.  If $p=1$ then $w_{1,A} \left( {T_1 , \cdots ,T_n } \right)$ (also, it is denoted by $w_{R,A} \left( {T_1 , \cdots ,T_n } \right)$) is called the
Rhombic $A$-numerical radius which have been studied in \cite{BAJMAEH} but for operators in $\mathscr{B} \left(\mathscr{H}\right)$. In an interesting case, $w_{1,A} \left( {C , \cdots ,C } \right)=n \cdot w_A \left( {C} \right)$.

The $A$-Crawford number is defined to be 
\begin{align*}
c_A\left( T \right) = \inf \left\{ {\left| \lambda\right|:\lambda
	\in W_A\left( T \right) } \right\} = \mathop {\inf }\limits_{\left\|
	x \right\|_A = 1} \left| {\left\langle {Tx,x} \right\rangle_A }
\right|.
\end{align*}
Consequently, we define the generalized $A$-Crawford number as:
\begin{align*} 
c_{p,A} \left( {T_1 , \cdots ,T_n } \right): = \mathop {\inf }\limits_{\left\| x \right\| = 1} \left( {\sum\limits_{i = 1}^n {\left| {\left\langle {T_i x,x} \right\rangle_A } \right|^p } } \right)^{1/p}, \qquad  p\ge1.
\end{align*}
In case $p=1$, the generalized Crawford number is called the Rhombic $A$-Crawford number and is denoted by $c_{R,A} \left( {T_1 , \cdots ,T_n } \right)$.

We note that in case $p=\infty$,  the generalized Euclidean operator radius is defined as:
\begin{align*} 
w_{\infty,A} \left( {T_1 , \cdots ,T_n } \right)&:= \mathop {\sup }\limits_{\left\| x \right\|_A = 1}  \sum\limits_{i = 1}^n {\left| {\left\langle {T_i x,x} \right\rangle_A } \right|  }  -\mathop {\inf }\limits_{\left\| x \right\| _A= 1}  \sum\limits_{i = 1}^n {\left| {\left\langle {T_i x,x} \right\rangle_A } \right|  }
\\
&=w_{R,A} \left( {T_1 , \cdots ,T_n } \right) -c_{R,A} \left( {T_1 , \cdots ,T_n } \right).
\end{align*}
Thus, the inequality
\begin{align}
\label{eq4.6} w_{\infty,A}  \left( {T_1 , \cdots ,T_n } \right) \le w_{p,A} \left( {T_1 , \cdots ,T_n } \right) \le w_{R,A} \left( {T_1 , \cdots ,T_n } \right)
\end{align}
for all $p\in \left(1,\infty\right)$. This fact follows by Jensen's inequality applied for the function $h(p) = w_{p,A} \left( {T_1 , \cdots ,T_n } \right)$, which is log-convex and decreasing for all $p > 1$.

On the other hand, 
by employing the Jensen's inequality
\begin{align*}
\left( {\frac{1}{n}\sum\limits_{k = 1}^n {a_k } } \right)^p \le \frac{1}{n}\sum\limits_{k = 1}^n {a^p_k  },
\end{align*} 
which holds for every finite positive sequence of real numbers $\left(a_k\right)_{k=1}^n$ and  $p\ge1$; by setting 
$a_k  = \left| {\left\langle {T_k x,x} \right\rangle_A } \right|$ for all $(k=1,2,\cdots,n)$, we get
\begin{align*}
\sum\limits_{k = 1}^n {\left| {\left\langle {T_k x,x} \right\rangle_A } \right|}  \le n^{1 - \frac{1}{p}} \left( {\sum\limits_{k = 1}^n {\left| {\left\langle {T_k x,x} \right\rangle_A } \right|^p } } \right)^{\frac{1}{p}}. 
\end{align*}
Taking the supremum over all $A$-unit vector $x\in \mathscr{H}$, one could get
\begin{align}
 w_{R,A} \left( {T_1 , \cdots ,T_n } \right) \le n^{1 - \frac{1}{p}} w_{p,A} \left( {T_1 , \cdots ,T_n } \right).\label{eq4.7}
\end{align}
Combining the inequalities \eqref{eq4.6} and \eqref{eq4.7} we get	
\begin{align}
w_{\infty,A} \left( {T_1 , \cdots ,T_n } \right) \le w_{p,A} \left( {T_1 , \cdots ,T_n } \right)\le w_{R,A} \left( {T_1 , \cdots ,T_n } \right) \le n^{1 - \frac{1}{p}} w_{p,A} \left( {T_1 , \cdots ,T_n } \right).\label{eq4.8} 
\end{align}	
More generally, in  the power mean inequality
\begin{align*}
\left( {\frac{1}{n}\sum\limits_{k = 1}^n {a_k^p } } \right)^{\frac{1}{p}}  \le \left( {\frac{1}{n}\sum\limits_{k = 1}^n {a_k^q } } \right)^{\frac{1}{q}}, \qquad \forall p\le q
\end{align*}
if one	chooses $a_k  = \left| {\left\langle {T_k x,x} \right\rangle_A } \right|$ for all $(k=1,2,\cdots,n)$, then we have
\begin{align*}
\left( {\frac{1}{n}\sum\limits_{k = 1}^n {\left| {\left\langle {T_k x,x} \right\rangle_A } \right|^p } } \right)^{\frac{1}{p}}  \le \left( {\frac{1}{n}\sum\limits_{k = 1}^n {\left| {\left\langle {T_k x,x} \right\rangle_A } \right|^q } } \right)^{\frac{1}{q}}. 
\end{align*}
Taking the supremum over all $A$-unit vector $x\in \mathscr{H}$, we  get	
\begin{align}
\label{eq4.9} w_{p,A} \left( {T_1 , \cdots ,T_n } \right) \le n^{\frac{1}{p} - \frac{1}{q}} w_{q,A} \left( {T_1 , \cdots ,T_n } \right), \qquad \forall q\ge p \ge 1.
\end{align}
Indeed, one can refine \eqref{eq4.8} by applying the Jensen's inequality
\begin{align}
 \left( {\frac{1}{n}\sum\limits_{k = 1}^n {a_k } } \right)^p  \le \frac{1}{n}\sum\limits_{k = 1}^n {a_k^p }  - \frac{1}{n}\sum\limits_{k = 1}^n {\left| {a_k  - \frac{1}{n}\sum\limits_{j = 1}^n {a_j } } \right|^p } \qquad p\ge2\label{eq4.10}
\end{align}
which obtained from more general result for superquadratic functions \cite{SJS}. 

Thus, by setting $a_k  = \left| {\left\langle {T_k x,x} \right\rangle_A } \right|$ in \eqref{eq4.10} we get
\begin{align*}
\left( {\sum\limits_{k = 1}^n {\left| {\left\langle {T_k x,x} \right\rangle_A } \right|} } \right)^p  &\le n^{p - 1} \sum\limits_{k = 1}^n {\left| {\left\langle {T_k x,x} \right\rangle_A } \right|^p }  - n^{p - 1} \sum\limits_{k = 1}^n {\left| {\left| {\left\langle {T_k x,x} \right\rangle_A } \right| - \frac{1}{n}\sum\limits_{j = 1}^n {\left| {\left\langle {T_j x,x} \right\rangle_A } \right|} } \right|^p } 
\\
&\le n^{p - 1} \sum\limits_{k = 1}^n {\left| {\left\langle {T_k x,x} \right\rangle_A } \right|^p }  - n^{p - 1} \sum\limits_{k = 1}^n {\left| {\left| {\left\langle {T_k x,x} \right\rangle_A } \right| - \frac{1}{n}\mathop {\sup }\limits_{\left\| x \right\| = 1} \sum\limits_{j = 1}^n {\left| {\left\langle {T_j x,x} \right\rangle_A } \right|} } \right|^p}. 
\end{align*}
Taking the supremum again over all $A$-unit vector $x\in \mathscr{H}$, we  get
\begin{align*}
&\mathop {\sup }\limits_{\left\| x \right\|_A = 1} \left( {\sum\limits_{k = 1}^n {\left| {\left\langle {T_k x,x} \right\rangle_A } \right|} } \right)^p \\ &\le \mathop {\sup }\limits_{\left\| x \right\|_A = 1} \left\{ {n^{p - 1} \sum\limits_{k = 1}^n {\left| {\left\langle {T_k x,x} \right\rangle_A } \right|^p }  - n^{p - 1} \sum\limits_{k = 1}^n {\left| {\left| {\left\langle {T_k x,x} \right\rangle_A } \right| - \frac{1}{n}\mathop {\sup }\limits_{\left\| x \right\|_A = 1} \sum\limits_{j = 1}^n {\left| {\left\langle {T_j x,x} \right\rangle_A } \right|} } \right|^p } } \right\} 
\\ 
&\le n^{p - 1} \mathop {\sup }\limits_{\left\| x \right\|_A = 1} \sum\limits_{k = 1}^n {\left| {\left\langle {T_k x,x} \right\rangle_A } \right|^p }  - n^{p - 1} \mathop {\inf }\limits_{\left\| x \right\|_A = 1} \sum\limits_{k = 1}^n {\left| {\left| {\left\langle {T_k x,x} \right\rangle_A } \right| - \frac{1}{n}\mathop {\sup }\limits_{\left\| x \right\|_A = 1} \sum\limits_{j = 1}^n {\left| {\left\langle {T_j x,x} \right\rangle_A } \right|} } \right|^p }  
\\ 
&= n^{p - 1} w_{p,A}^p \left( {T_1 , \cdots ,T_n } \right) - n^{p - 1} \mathop {\inf }\limits_{\left\| x \right\|_A = 1} \sum\limits_{k = 1}^n {\left| {\left| {\left\langle {T_k x,x} \right\rangle_A } \right| - \frac{1}{n}w_{R,A} \left( {T_1 , \cdots ,T_n } \right)} \right|^p },
\end{align*}
which gives
\begin{align*}
w_{R,A}^p \left( {T_1 , \cdots ,T_n } \right) 
\le n^{p - 1} w_{p,A}^p \left( {T_1 , \cdots ,T_n } \right) - n^{p - 1} \mathop {\inf }\limits_{\left\| x \right\|A_ = 1} \sum\limits_{k = 1}^n {\left| {\left| {\left\langle {T_k x,x} \right\rangle_A } \right| - \frac{1}{n}w_{R,A} \left( {T_1 , \cdots ,T_n } \right)} \right|^p }. 
\end{align*} 
which refine the right hand side of \eqref{eq4.8}. Clearly, all  above mentioned inequalities generalize and refine some inequalities obtained in \cite{MSS}. 
For recent inequalities, counterparts, refinements and other related properties concerning  the generalized Euclidean operator radius 
the reader my refer to \cite{BAJMAEH}, \cite{D4} ,\cite{HLB1},\cite{HLB2}, \cite{SMY}, \cite{SMS}, and \cite{P}.

Next, we give a generalization of \eqref{eq4.3} and refine (indeed improve) \eqref{eq4.2} (and thus \eqref{eq4.1}) to the generalized Euclidean operator radius.

\begin{theorem}
	Let $T_k\in \mathscr{B}_A\left(\mathscr{H}\right)$ $(k=1,\cdots,n)$. Then	
	\begin{align}
	\label{eq4.11} \frac{1}{{2^{p+1}n^{p - 1} }}\left\| {\sum\limits_{k = 1}^n {T_k^{\sharp_A} T_k  + T_k T_k^{\sharp_A} } } \right\|_A^p \le w^p_{2p,A} \left( {T_1 , \cdots ,T_n } \right) \le \frac{1}{{2^p }}\left\| {\sum\limits_{k = 1}^n {\left( {T_k^{\sharp_A} T_k  + T_k T_k^{\sharp_A} } \right)^p } } \right\|_A
	\end{align}
	for all $p\ge1$.
\end{theorem}
\begin{proof}
	Let $B_k+iC_k$ be the $A$-Cartesian decomposition of $T_k$ for all $k=1,\cdots,n$.
	Then, we have
	\begin{align*}
	\left| {\left\langle {T_k x,x} \right\rangle_A } \right|^{2p}  &= \left( {\left\langle {B_k x,x} \right\rangle_A ^2  + \left\langle {C_k x,x} \right\rangle_A ^2 } \right)^p  \\ 
	&\ge \frac{1}{{2^p }}\left( {\left| {\left\langle {B_k x,x} \right\rangle_A } \right| + \left| {\left\langle {C_k x,x} \right\rangle_A } \right|} \right)^{2p}  \\ 
	&\ge \frac{1}{{2^p }}\left| {\left\langle {B_k x,x} \right\rangle_A  + \left\langle {C_k x,x} \right\rangle_A } \right|^{2p}  \\ 
	&= \frac{1}{{2^p }}\left| {\left\langle {B_k  \pm C_k x,x} \right\rangle_A } \right|^{2p}.    
	\end{align*} 
	Summing over $k$ and then taking the supremum over all $A$-unit vector $x\in \mathscr{H}$, we get
	\begin{align*}
	w^p_{2p,A} \left( {T_1 , \cdots ,T_n } \right) &\ge \frac{1}{{2^p }}\mathop {\sup }\limits_{\left\| x \right\|_A = 1} \sum\limits_{k = 1}^n {\left| {\left\langle {B_k  \pm C_k x,x} \right\rangle_A } \right|^{2p} }  
	\\
	&\ge \frac{1}{{2^p }}\frac{1}{{n^{p - 1} }}\mathop {\sup }\limits_{\left\| x \right\|_A = 1} \left( {\sum\limits_{k = 1}^n {\left| {\left\langle {B_k  \pm C_k x,x} \right\rangle_A } \right|^2 } } \right)^p  \qquad  (\text{by Jensen's inequality})\\ 
	&= \frac{1}{{2^p }}\frac{1}{{n^{p - 1} }}\left\| {\sum\limits_{k = 1}^n {\left( {B_k  \pm C_k } \right)^2 } } \right\|_A^p.  
	\end{align*}
	Thus,
	\begin{align*}
	2w^p_{2p,A} \left( {T_1 , \cdots ,T_n } \right) &\ge \frac{1}{{2^p }}\frac{1}{{n^{p - 1} }}\left\| {\sum\limits_{k = 1}^n {\left( {B_k  + C_k } \right)^2 } } \right\|^p  + \frac{1}{{2^p }}\frac{1}{{n^{p - 1} }}\left\| {\sum\limits_{k = 1}^n {\left( {B_k  - C_k } \right)^2 } } \right\|_A^p  \\ 
	&\ge \frac{1}{{2^p }}\frac{1}{{n^{p - 1} }}\left\| {\sum\limits_{k = 1}^n {\left( {B_k  + C_k } \right)^2 }  + \sum\limits_{k = 1}^n {\left( {B_k  - C_k } \right)^2 } } \right\|_A^p  
	\\ 
	&= \frac{1}{{2^p }}\frac{1}{{n^{p - 1} }}\left\| {\sum\limits_{k = 1}^n {\left\{ {\left( {B_k  + C_k } \right)^2  + \left( {B_k  - C_k } \right)^2 } \right\}} } \right\|_A^p 
	 \\ 
	&=  \frac{1}{{n^{p - 1} }}\left\| {\sum\limits_{k = 1}^n {B_k^2  + C_k^2 } } \right\|_A^p  \\ 
	&=     \frac{1}{{n^{p - 1} }}\left\| {\sum\limits_{k = 1}^n {\frac{T_k^{\sharp_A} T_k  + T_k T_k^{\sharp_A}}{2} } } \right\|_A^p
	\\ 
	&=   \frac{1}{{2^pn^{p - 1} }}\left\| {\sum\limits_{k = 1}^n { T_k^{\sharp_A} T_k  + T_k T_k^{\sharp_A}  } } \right\|_A^p,  
	\end{align*}
	and hence,
	\begin{align*}
	w^p_{2p} \left( {T_1 , \cdots ,T_n } \right)   \ge    \frac{1}{{2^{p+1}n^{p - 1} }}\left\| {\sum\limits_{k = 1}^n {T_k^{\sharp} T_k  + T_k T_k^{\sharp} } } \right\|_A^p,  
	\end{align*}
	which proves the left hand side of the inequality in \eqref{eq4.11}.

	To prove the second inequality, for every $A$-unit vector $x\in \mathscr{H}$ we have
	\begin{align*}
	\sum\limits_{k = 1}^n {\left| {\left\langle {T_k x,x} \right\rangle_A } \right|^{2p} }  &= \sum\limits_{k = 1}^n {\left( {\left\langle {B_k x,x} \right\rangle_A ^2  + \left\langle {C_k x,x} \right\rangle_A ^2 } \right)^p }  \\ 
	&\le \sum\limits_{k = 1}^n {\left( {\left\langle {B_k^2 x,x} \right\rangle_A  + \left\langle {C_k^2 x,x} \right\rangle_A } \right)^p }  \\ 
	&= \sum\limits_{k = 1}^n {\left\langle {\left( {B_k^2  + C_k^2 } \right)x,x} \right\rangle_A ^p },
	\end{align*}
	which implies that 
	\begin{align*} 
	\mathop {\sup }\limits_{\left\| x \right\|_A = 1}\sum\limits_{k = 1}^n {\left| {\left\langle {T_k x,x} \right\rangle } \right|^{2p} } =w_{2p,A}^p \left( {T_1 , \cdots ,T_1 } \right)&\le \mathop {\sup }\limits_{\left\| x \right\|_A = 1} \sum\limits_{k = 1}^n {\left\langle {\left( {B_k^2  + C_k^2 } \right)x,x} \right\rangle_A ^p }  
	\\
	&= \left\| {\sum\limits_{k = 1}^n {\left( {B_k^2  + C_k^2 } \right)^p } } \right\|_A 
	\\
	&= \frac{1}{{2^p }}\left\| {\sum\limits_{k = 1}^n {\left( {T_k^{\sharp_A} T_k  + T_k T_k^{\sharp_A} } \right)^p } } \right\|_A,
	\end{align*}
	which proves the right hand side of \eqref{eq4.11}.
\end{proof}

\begin{remark}
	Clearly, by setting $n=1$ and $p=1$ in \eqref{eq4.11} we recapture \eqref{eq4.2}.
\end{remark}
A very interesting case of \eqref{eq4.11} is considered in the following corollary.
\begin{corollary}
	Let $T,S\in \mathscr{B}\left(\mathscr{H}\right)$. Then	
	\begin{align}
	\label{eq4.12}\frac{1}{{2^{2p} }}\left\| {  {T^{\sharp_A} T  + T T^{\sharp_A} +S^{\sharp_A} S + S S^{\sharp_A} } } \right\|_A^p &\le w^p_{2p,A} \left( {T,S} \right) 
	\\
	&\le \frac{1}{{2^p }}\left\| {  {\left( {T^{\sharp_A} T  + T T^* } \right)^p +\left( {S^{\sharp_A} S + S S^{\sharp_A}  } \right)^p} } \right\|_A\nonumber
	\end{align}
	for all $p\ge1$. 
\end{corollary}
\begin{proof}
	Setting $n=2$ in \eqref{eq4.11}. 
\end{proof}

\begin{remark}
	In particular, setting $p=1$ in \eqref{eq4.12} we get
	\begin{align*}
	\frac{1}{4} \left\| {  {T^{\sharp_A} T  + T T^{\sharp_A} +S^{\sharp_A} S + S S^{\sharp_A}} } \right\|_A  &\le w_{\rm{e},A} \left( {T,S} \right) 
	\\
	&\le \frac{1}{{2  }}\left\| {T^{\sharp_A} T  + T T^{\sharp_A}+ S^{\sharp_A} S + S S^{\sharp_A}   } \right\|_A.
	\end{align*}
	Moreover, if we choose $T=S$, then
	\begin{align*}
	\frac{1}{2} \left\| {  {T^{\sharp_A} T  + T T^{\sharp_A}  } } \right\|_A   \le w_{\rm{e},A} \left( {T,T} \right) 
	\le  \left\| {T^{\sharp_A} T  + T T^{\sharp_A}} \right\|_A.
	\end{align*}
\end{remark}

\begin{remark}
	A lower and  upper bounds for the Rhombic numerical radius could be deduced as follows:
	
	In \eqref{eq4.8} the inequality holds for any $p\ge1$. Setting $p=2q$, then \eqref{eq4.8} reduces to  
	\begin{align*}
	w_{2q,A} \left( {T_1 , \cdots ,T_n } \right)\le w_{R,A} \left( {T_1 , \cdots ,T_n } \right) \le n^{1 - \frac{1}{2q}} w_{2q,A} \left( {T_1 , \cdots ,T_n } \right).
	\end{align*}
	which implies that
	\begin{align}
	w^q_{2q,A} \left( {T_1 , \cdots ,T_n } \right)\le w^q_{R,A} \left( {T_1 , \cdots ,T_n } \right) \le n^{q - \frac{1}{2}} w^q_{2q,A} \left( {T_1 , \cdots ,T_n } \right). \label{eq4.13}
	\end{align} 
	Combining the inequalities \eqref{eq4.11} with \eqref{eq4.13} we get	
	\begin{align*}
	\frac{1}{{2^{2q+1}n^{q - 1} }}\left\| {\sum\limits_{k = 1}^n {T_k^{\sharp_A} T_k  + T_k T_k^{\sharp_A} } } \right\|_A^q 
	&\le w^q_{2q,A} \left( {T_1 , \cdots ,T_n } \right) 
	\\
	&\le w^q_{R,A} \left( {T_1 , \cdots ,T_n } \right) 
	\\
	&\le n^{q - \frac{1}{2}} w^q_{2q,A} \left( {T_1 , \cdots ,T_n } \right)
	\\
	&\le \frac{n^{q - \frac{1}{2}}}{{2^q }}\left\| {\sum\limits_{k = 1}^n {\left( {T_k^{\sharp_A}T_k  + T_k T_k^{\sharp_A} } \right)^q } } \right\|_A 
	\end{align*}
	for any $q\ge \frac{1}{2}$.
\end{remark}




\end{document}